\documentclass[11pt]{article}

\usepackage[letterpaper,top=2cm,bottom=2cm,left=3cm,right=3cm,marginparwidth=1.75cm]{geometry}

\usepackage{algpseudocode}
\usepackage{amsfonts} 
\usepackage{amsmath}
\usepackage{amsopn}
\usepackage{amssymb}
\usepackage{amsthm}
\usepackage{bm}
\usepackage{booktabs}
\usepackage{caption}
\usepackage{commath}
\usepackage{csvsimple}
\usepackage{enumitem}
\usepackage[T1]{fontenc}
\usepackage{hyperref}
\usepackage{mathtools}
\usepackage{subcaption}
\usepackage{tikz}
\usepackage{lipsum}
\usepackage{thmtools}
\usepackage{slantsc}
\usepackage{nccmath}
\usepackage{svg}
\usepackage{cleveref}

\graphicspath{{figures/}}

\newtheorem{example}{\upshape Example}
\newtheorem{assumption}{\upshape Assumption}
\newtheorem{remark}{\upshape Remark}
\newtheorem{property}{\upshape Property}
\newtheorem{theorem}{\upshape Theorem}
\newtheorem{lemma}{\upshape Lemma}
\newtheorem{corollary}{\upshape Corollary}
\crefname{example}{example}{examples}
\crefname{assumption}{assumption}{assumptions}
\crefname{remark}{remark}{remarks}
\crefname{property}{property}{properties}
\crefname{theorem}{theorem}{theorems}
\crefname{lemma}{lemma}{lemmas}
\crefname{corollary}{corollary}{corollaries}

\DeclarePairedDelimiter\normZZZZ{\lVert}{\rVert}
\DeclarePairedDelimiter\absZZZZ{\lvert}{\rvert}

\DeclareSymbolFont{bbold}{U}{bbold}{m}{n}
\DeclareSymbolFontAlphabet{\mathbbold}{bbold}

\newcommand{\floor}[1]{\left\lfloor #1 \right\rfloor} 

\begin{document}

\title{Single-ensemble multilevel Monte Carlo for discrete ensemble Kalman methods}
\author{Arne Bouillon\thanks{NUMA research group, Department of Computer Science, KU Leuven, Leuven, Belgium (\href{mailto:arne.bouillon@kuleuven.be}{arne.bouillon@kuleuven.be}, \href{mailto:toon.ingelaere@kuleuven.be}{toon.ingelaere@kuleuven.be}, \href{mailto:giovanni.samaey@kuleuven.be}{giovanni.samaey@kuleuven.be})} \and Toon Ingelaere\footnotemark[1] \and Giovanni Samaey\footnotemark[1]}
\maketitle

\begin{abstract}
    Ensemble Kalman methods solve problems in domains such as filtering and inverse problems with interacting particles that evolve over time. For computationally expensive problems, the cost of attaining a high accuracy quickly becomes prohibitive. We exploit a hierarchy of approximations to the underlying forward model and apply multilevel Monte Carlo (MLMC) techniques, improving the asymptotic cost-to-error relation. More specifically, we use MLMC \emph{at each time step} to estimate the interaction term in a single, globally-coupled ensemble. This technique was proposed by Hoel et al.\ for the ensemble Kalman filter; our goal is to study its applicability to a broader family of ensemble Kalman~methods.

\end{abstract}

\section{Introduction} \label{sec:intro}
    This paper studies ensemble Kalman methods, algorithms that solve various problems with an evolving \emph{ensemble} of interacting \emph{particles} in state or parameter space. These have been particularly successful in the contexts of filtering \cite{andersonEnsembleAdjustmentKalman2001,bergemannEnsembleKalmanBucyFilter2012,bishopAdaptiveSamplingEnsemble2001,evensenSequentialDataAssimilation1994a}, optimization \cite{iglesiasEnsembleKalmanMethods2013a,schillingsAnalysisEnsembleKalman2017b}, rare-event estimation \cite{wagnerEnsembleKalmanFilter2022a}, and Bayesian-posterior sampling \cite{garbuno-inigoInteractingLangevinDiffusions2020a,iglesiasIterativeRegularizationEnsemble2015}. Some example methods are introduced in \cref{sec:intro:ipm}. With a finite number of particles, they can be viewed as Monte Carlo approximations to some mean-field model.

    Our work compares this straightforward approximation to a newly proposed generalization of the multilevel Monte Carlo (MLMC) scheme from \cite{chernovMultilevelEnsembleKalman2021a,hoelMultilevelEnsembleKalman2016a}, where pairs of particles follow different but still globally-coupled dynamics.

    \subsection{Ensemble Kalman methods} \label{sec:intro:ipm}
        This subsection discusses the use of ensemble Kalman methods for filtering, as well as optimization and sampling in Bayesian inverse problems. Filtering is concerned with reconstructing state variables from noisy observations. Consider the discrete dynamics
        \begin{equation} \label{eq:intro:ipm:filter-dyn}
            u_{n+1} = \mathcal G(u_n), \qquad y_{n+1} = Hu_{n+1} + \eta_{n+1}, \qquad 0 \le n < N,\\
        \end{equation}
        where $u_n\in\mathbb R^{d_u}$ denotes the state at time step $n$, $\mathcal G$ is the stochastic \emph{forward model}, $H$ is a linear observation map, and $\eta_n$ is a noise term. One popular algorithm to estimate the states $\{u_n\}_{n=1}^N$ from noisy observations $\{y_n\}_{n=1}^N$ is the ensemble Kalman filter (EnKF).

        \begin{example}[Ensemble Kalman filter] \label{ex:intro:ipm:enkf}
            The EnKF \cite{evensenSequentialDataAssimilation1994a} is an ensemble Kalman method whose ensemble $\bm u_n = \{u_n^j\}_{j=1}^J$ at time $n$ estimates the expectation and uncertainty on $u_n$. It assumes that $\eta_n\sim\mathcal N(0, \Gamma)$ with positive definite $\Gamma$. A particle $u_n^j$ follows
            \begin{equation} \label{eq:ex:intro:ipm:enkf}
                u_{n+1}^j = (I - K^{\mathcal G}(\bm u_n)H)\mathcal G(u_n^j) + K^{\mathcal G}(\bm u_n)(y_{n+1} + \sqrt\Gamma\xi_n^j),
            \end{equation}
            where $K^{\mathcal G}(\bm u_n) = C(\mathcal G(\bm u_n))H^{\mathstrut\scriptstyle{\top}}(HC(\mathcal G(\bm u_n))H^{\mathstrut\scriptstyle{\top}} + \Gamma)^{-1}$ (with $C(\cdot)$ the sample covariance) is called the \emph{Kalman gain}, and where $\xi_n^j\sim\mathcal N(0, I)$.
        \end{example}
        \begin{example}[Deterministic ensemble Kalman filter] \label{ex:intro:ipm:denkf}
            The deterministic ensemble Kalman filter (DEnKF) is proposed in \cite{sakovDeterministicFormulationEnsemble2008} as an alternative to the EnKF. It uses the dynamics
            \begin{equation} \label{eq:ex:intro:ipm:denkf}
                u_{n+1}^j = (I - K^{\mathcal G}(\bm u_n)H)\mathcal G(u_n^j) + K^{\mathcal G}(\bm u_n)(y_{n+1} + H/2(\mathcal G(u_n^j) - E(\mathcal G(\bm u_n)))),
            \end{equation}
            where $K^{\mathcal G}(\bm u_n)$ is the Kalman gain in \cref{ex:intro:ipm:enkf}. $E(\cdot)$ denotes the sample mean.
        \end{example}

        A second problem class is that of \emph{Bayesian inverse problems}. Here we assume to have an unknown parameter $u\in\mathbb R^{d_u}$ with prior distribution $\pi_\mathrm{prior}(u)$, a deterministic forward map $\mathcal G\colon \mathbb R^{d_u}\rightarrow\mathbb R^{d_g}$, and an observation
        \begin{equation}
            y = \mathcal G(u) + \eta,
        \end{equation}
        in which $\eta\in\mathbb R^{d_g}$ follows a known noise model $\pi_\eta$. We can then define the \emph{likelihood} $\pi_\mathrm{li}(y\mid u) \coloneqq \pi_\eta(y - \mathcal G(u))$. Bayes' formula results in the posterior distribution
        \begin{equation} \label{eq:intro:ipm:posterior}
            \pi_\mathrm{post}(u \mid y) \propto \pi_\mathrm{li}(y \mid u)\pi_\mathrm{prior}(u),
        \end{equation}
        of the unknown parameter $u$. Computing the normalization constant is usually intractable, as it involves integration over the entire parameter domain.

        The posterior distribution is mainly used in two ways. Optimization methods can target the \emph{maximum a posteriori} (MAP) parameter, the most likely $u$ given~$y$ and $\pi_\mathrm{prior}$. Sampling methods give a more complete view of the posterior and its features by sampling from it. Ensemble Kalman inversion (EKI) \cite{iglesiasEnsembleKalmanMethods2013a} and ensemble Kalman sampling (EKS) \cite{garbuno-inigoInteractingLangevinDiffusions2020a} perform these respective tasks and are both inspired by the EnKF.
        \begin{example}[Ensemble Kalman inversion] \label{ex:intro:ipm:eki}
            We assume that $\eta\sim\mathcal N(0, \Gamma)$ and that $u$ has a uniform prior. (General noise distributions are handled in~\cite{duffieldEnsembleKalmanInversion2022a}; prior regularization is discussed in, e.g., \cite{huangIteratedKalmanMethodology2022b}.) EKI was proposed as the iterated application of the EnKF (creating an artificial discrete time dimension) in \cite{iglesiasEnsembleKalmanMethods2013a}, to which time steps $\tau_n$ were added in \cite{schillingsAnalysisEnsembleKalman2017b}. With $\xi_n^j\sim\mathcal N(0, I)$, the resulting dynamics~are
            \begin{equation} \label{eq:ex:intro:ipm:eki}
                u_{n+1}^j = u_n^j + \tau_n C(\bm u_n, \mathcal G(\bm u_n))(\tau_n C(\mathcal G(\bm u_n)) + \Gamma)^{-1}(y - \mathcal G(u_n^j) + \sqrt{\Gamma / \tau_n}\kern2pt\xi_n^j),
            \end{equation}
            again with sample (cross-)covariance $C(\cdot)$. A continuous-time limit was studied in \cite{schillingsAnalysisEnsembleKalman2017b} and rediscretized in a slightly different form in, e.g., \cite{kovachkiEnsembleKalmanInversion2019a}. In that work, the (artificial) time steps $\tau_n$ are also determined adaptively.
        \end{example}
        \begin{example}[Ensemble Kalman sampling] \label{ex:intro:ipm:eks}
            Now assume that $\eta\sim\mathcal N(0, \Gamma)$ and that $\pi_\mathrm{prior}$ is a zero-centered Gaussian with covariance $\Gamma_0$. EKS was proposed and motivated in continuous-time form in \cite{garbuno-inigoInteractingLangevinDiffusions2020a}. In practice, a discretization should be used, such as
            \begin{equation} \label{eq:ex:intro:ipm:eks}
                u_{n+1}^{j} = u_n^j + \tau_nC(\bm u_n, \mathcal G(\bm u_n))\Gamma^{-1}(y - \mathcal G(u_n^j)) - \tau_nC(\bm u_n)\Gamma_0^{-1}u_{n+1}^{j} + \sqrt{2\tau_nC(\bm u_n)}\kern2pt\xi_n^j\\
            \end{equation}
            with $\xi_n^j\sim\mathcal N(0, I)$. These dynamics estimate (\ref{eq:intro:ipm:posterior}), based on linear ansatzes, as $n\rightarrow\infty$.
        \end{example}

        \paragraph{Advantages of ensemble Kalman methods.} Many of these methods require no derivatives of the forward model; instead of gradient information, interaction between the ensemble members drives the particle evolution. This is crucial when gradients are expensive, unavailable, or undefined due to a non-differentiable objective \cite{kovachkiEnsembleKalmanInversion2019a}, or when they are noisy or highly oscillatory \cite{dunbarEnsembleInferenceMethods2022a}. In addition, these methods allow for straightforward parallelization, as only the interaction term requires information from multiple particles.

    \subsection{Multilevel Monte Carlo} \label{sec:intro:mlmc}
        To simulate ensemble Kalman methods with expensive models more efficiently, we will use multilevel Monte Carlo (MLMC) \cite{gilesMultilevelMonteCarlo2008b}. The core MLMC idea is as follows. An expectation $\mathbb E[x_L]$ of an expensive random variable $x_L$, to which a hierarchy of cheaper, less accurate approximations $\{x_\ell\}_{\ell=0}^{L-1}$ is available, is rewritten with a telescoping sum:
        \begin{equation} \label{eq:intro:mlmc:mlmc}
            \mathbb E[x_L] = \mathbb E[x_0] + \medmath{\sum\nolimits}_{\ell=1}^L\mathbb E[x_\ell - x_{\ell-1}].
        \end{equation}
        MLMC samples many cheap realizations of $x_0$, giving an accurate estimate of $\mathbb E[x_0]$. Each difference term is then estimated by sampling \emph{correlated} realizations of $x_\ell$~and~$x_{\ell-1}$. This correlation reduces the variance of the estimators, so fewer samples are needed. The challenge in designing MLMC algorithms is to find a way to correlate these realizations.

    \subsection{Related work and objectives} \label{sec:intro:obj}
        Multilevel methods for filtering \cite{gregoryMultilevelEnsembleTransform2016,jasraMultilevelParticleFilters2017} and Bayesian inversion \cite{dodwellMultilevelMarkovChain2019f} are an active research topic. Within ensemble Kalman methods, a multilevel EnKF was proposed in \cite{hoelMultilevelEnsembleKalman2016a} and extended to spatio-temporal processes in \cite{chernovMultilevelEnsembleKalman2021a}. A variant for reservoir history matching is given in \cite{fossumAssessmentMultilevelEnsemblebased2020} and a multifidelity EnKF in \cite{popovMultifidelityEnsembleKalman2021}. We will refer to these algorithms as \emph{single-ensemble} MLMC, as they use \emph{a sole ensemble} of pairwise-correlated particles -- with fewer particle pairs on higher levels -- that \emph{interact globally}.

        An alternative approach is developed in \cite{hoelMultilevelEnsembleKalman2020b} and given a multi-index extension~in~\cite{hoelMultiindexEnsembleKalman2022}. They use many small, inaccurate ensembles together with fewer large, accurate ones. All ensembles evolve independently; we will call these methods \emph{multiple-ensemble} MLMC.

        These particle systems are closely related to McKean--Vlasov SDEs, whose~evolution depends on the law of the solution. In this context, many multilevel ideas~are found in the literature \cite{chadaMultilevelEnsembleKalman2022b,haji-aliMultilevelMultiindexMonte2018a,ricketsonMultilevelMonteCarlo2015,szpruchIterativeMultilevelParticle2019a} and inspired the multilevel methods above. Of these,~\cite{ricketsonMultilevelMonteCarlo2015} comes closest to the single-ensemble approach, but uses less coupling between levels~and focuses specifically on the expectation of a function over the particles as interaction.

        There are, however, key differences between the general McKean--Vlasov case and ensemble Kalman methods. McKean--Vlasov MLMC techniques often vary the time step used between levels, while ensemble Kalman methods either do not have time steps or, adaptively \cite{kovachkiEnsembleKalmanInversion2019a}, tend to choose the largest time step that does not cause instabilities. In addition, the interaction terms in ensemble Kalman methods are typically means or covariances, which have cost $\mathcal O(J)$ with $J$ particles instead of the $\mathcal O(J^2)$ in many other particle systems \cite{carrilloConsensusBasedSampling2022a,pinnauConsensusbasedModelGlobal2017}. Both of these properties support a single-ensemble approach: particles on all levels are defined at each time step and global interaction is cheap.

        In \cite{hoelMultilevelEnsembleKalman2020b}, the multiple-ensemble multilevel EnKF is compared to the single-ensemble one \pagebreak{}from \cite{hoelMultilevelEnsembleKalman2016a} for some test problems. This shows the latter approach consistently outperforming the former by a constant factor. Nevertheless, single-ensemble multilevel ensemble Kalman methods remain restricted to the EnKF. Our goal, then, is twofold:~(i)~formulate a framework for ensemble Kalman methods with a single-ensemble multilevel simulation algorithm, and (ii) analyze the rate at which single- and multilevel simulation algorithms converge to the mean-field model when more particles are added.

    \subsection{Overview of the paper} \label{sec:intro:overview}
        After \cref{sec:notation} introduces our notation, we formulate a general framework for MLMC ensemble Kalman methods in \cref{sec:frame}. \Cref{sec:conv} studies the asymptotic cost-to-error relation of this technique, with proofs deferred to \cref{sec:proof-sl,sec:proof-ml}. The performance~of~our algorithm is studied numerically in \cref{sec:scale}, after which \cref{sec:concl} concludes the paper.

\section{Notation and prerequisites} \label{sec:notation}
        Let $(\Omega, \mathcal F, \mathbb P)$ be a complete probability space. For any $d\in\mathbb N$ and $p\ge2$, the $p$-norm of a random variable~(RV) ${u\colon\Omega\rightarrow\mathbb R^d}$ is defined as
        \begin{equation}
            \normZZZZ{u}_p \coloneqq \mathbb E[\absZZZZ u^p]^{1/p}.
        \end{equation}
        For a scalar $u$, $\absZZZZ u$ denotes the absolute value; for a vector or matrix $u$, it denotes the 2-norm; and for a tuple $u$, it denotes the sum of the element norms. We introduce the space $L^p(\Omega, \mathbb R^d) \coloneqq \{u\colon \Omega\rightarrow\mathbb R^d \mid \normZZZZ{u}_p < \infty\}$. We will also use the shorthand notation $L^{\ge2}(\Omega, \mathbb R^d) \coloneqq \bigcap_{p\ge2}L^p(\Omega, \mathbb R^d)$. The following properties will prove useful.

        \begin{property}[Generalized H\"older's inequality] \label{pr:intro:notation:holder}
            For any RVs $(u, v)$ and $p\ge2$, we have $\normZZZZ{uv}_p \le \normZZZZ{u}_q\normZZZZ{v}_r$ if $1/p = 1/q + 1/r$. In particular, $\normZZZZ{uv}_p \le \normZZZZ{u}_{2p}\normZZZZ{v}_{2p}$.
        \end{property}

        \begin{property}[Norm ordering] \label{pr:intro:notation:ord}
            For any RV $u$ and $p\ge2$, we have ${\absZZZZ{\mathbb E[u]} \le \mathbb E[\absZZZZ u] \le \normZZZZ{u}_p}$.
        \end{property}

        \begin{property}[Monotonicity of the $p$-norm] \label{pr:intro:notation:mono}
            For any RVs $(u, v)$ and $p\ge2$, if it holds that ${\absZZZZ{u(\omega)} \le \absZZZZ{v(\omega)}}$ for all $\omega\in\Omega$, then $\normZZZZ u_p \le \normZZZZ v_p$.
        \end{property}

        \begin{property}[Marcinkiewicz--Zygmund inequality] \label{pr:intro:notation:mz}
            Let $u$ and $u^1, \ldots, u^J$ be zero-mean i.i.d.\ RVs such that $\normZZZZ{u}_p<\infty$ for all $p\ge2$. Then, for any $p\ge2$, there exists a constant $c_{\kern-1ptp}$ such that $\normZZZZ{\frac1J\sum_{j=1}^J u^j}_p \le c_{\kern-1ptp}J^{-1/2}\normZZZZ{u}_p$. (See, e.g., \cite[Corollary 8.2]{gutProbabilityGraduateCourse2013}.)
        \end{property}

        We write $A\succ0$ (or $A\succeq0$) to indicate that a matrix $A$ is positive (semi-)definite. The expressions $A\succ B$ and $A\succeq B$ mean $A-B\succ0$ and $A-B\succeq0$, respectively. The notation $f(x) \lesssim g(x)$ will denote that there exists a constant $c$ such that $f(x) \le cg(x)$ for all $x$. We further write $f(x) \eqsim g(x)$ to mean ${f(x) \lesssim g(x) \lesssim f(x)}$.

\section{Presentation of the framework} \label{sec:frame}
    We now present our framework for ensemble Kalman methods. For many practical problems, the model $\mathcal G$ is computationally intractable. Instead, a hierarchy of approximations $\{\mathcal G_\ell\}_{\ell=0}^\infty$ is available, where a higher $\ell$ offers a better approximation. We work in this context. First, \cref{sec:frame:sl-mf} identifies a common structure to the methods introduced so far that approximates underlying \emph{mean-field dynamics}. \Cref{sec:frame:ml} then proposes a multilevel simulation algorithm that approximates the mean-field model with the hierarchy~$\{\mathcal G_\ell\}_\ell$.

    \subsection{Single-level simulation algorithm} \label{sec:frame:sl-mf}
        We now discuss how the dynamics in \cref{ex:intro:ipm:enkf,ex:intro:ipm:denkf,ex:intro:ipm:eki,ex:intro:ipm:eks} can be interpreted as particle discretisations of a mean-field discrete-time McKean--Vlasov-type equation, with initial condition $u_0\in L^{\ge2}(\Omega, \mathbb R^{d_u})$. A mean-field particle taking $N$ time steps is a realization of the correlated random variables $\{\bar u_n\colon\Omega\rightarrow\mathbb R^{d_u}\}_{n=0}^N$. The particle evolves over time~as
        \begin{equation} \label{eq:frame:def:mf}
                \bar u_{n+1}(\omega) = \Psi_n^{\mathcal G(\cdot, \omega)}(\bar u_n(\omega), \Theta^{\mathcal G}[\bar u_n], \xi_n(\omega)),
        \end{equation}
        where $\xi_n\sim\mathcal N(0, I)$ and $\Theta^g[u] = (\Theta_1^g[u], ..., \Theta_M^g[u])$ contains $M$ statistical parameters of a random variable $u$ and may involve a forward model $g$. Usually one is interested in a quantity of interest (QoI) $\bar\theta^\dagger_N \coloneq \Theta^\dagger[\bar u_N]$, some parameter of the distribution at time~$N$.

        We estimate this distribution of $\bar u_N$ through $J$ approximate samples from \cref{eq:frame:def:mf}. Let $\omega^j\in\Omega$ for $1\le j\le J$ and consider the level-$L$ and $J$-particle ensemble $\bm u_n^L = \{u_n^{L,j}\}_{j=1}^J$:
        \begin{equation} \label{eq:frame:def:sl}
            u_{n+1}^{L,j} = \Psi_n^{\mathcal G_L}(u_n^{L,j}, \widehat\Theta^{\mathcal G_L}(\bm u_n^L), \xi_n^j), \qquad 1\le j\le J.
        \end{equation}
        The \emph{sample statistic} $\widehat \Theta^g(\bm u)$ estimates $\Theta^g[u]$ with an ensemble $\bm u$, distributed as $u$. In \cref{eq:frame:def:sl}, we defined $\xi_n^j\coloneqq\xi_n(\omega^j)$. Note also that when a forward model $g$ is stochastic, $g(u_n^{L,j})$ should be interpreted as $g(u_n^{L,j}, \omega^j)$ in the computation of $\Psi_n^g$ and $\widehat\Theta^g$. We refer to \cref{eq:frame:def:sl} as the \emph{single-level} simulation algorithm, as it employs a single approximation from the hierarchy $\{\mathcal G_\ell\}_\ell$. The QoI $\bar\theta^\dagger_N$ is estimated via an estimator $\hat \theta^{\dagger,L}_N \coloneq \widehat \Theta^\dagger(\bm u^L_N)$.

        \begin{remark}
            The dynamics in \cref{ex:intro:ipm:enkf,ex:intro:ipm:denkf,ex:intro:ipm:eki,ex:intro:ipm:eks} fit this framework with sample means $E(\cdot)$ and sample covariances $C(\cdot)$ as statistics. Convergence to their mean-field limit as ${J\rightarrow\infty}$, in either discrete-time or continuous-time form, is studied in e.g. \cite{dingEnsembleKalmanSampler2021a,garbuno-inigoInteractingLangevinDiffusions2020a,mandelConvergenceEnsembleKalman2011b,schillingsAnalysisEnsembleKalman2017b}.
        \end{remark}

    \subsection{Multilevel simulation algorithm} \label{sec:frame:ml}
        \Cref{eq:frame:def:sl} is a straightforward Monte Carlo approximation to \cref{eq:frame:def:mf} with fixed accuracy level $L$ and ensemble size $J$. In contrast, the single-ensemble multilevel Monte Carlo approach mixes particles on different levels $0\le\ell\le L$ with a multilevel ensemble that consists of subensembles: $\bm u_n^{\mathrm{ML}} = (\bm u_n^{0,{\mathrm F}}, (\bm u_n^{1,{\mathrm F}},\bm u_n^{1,{\mathrm C}}), \ldots, (\bm u_n^{L,{\mathrm F}},\bm u_n^{L,{\mathrm C}}))$. This approach follows \cite{chernovMultilevelEnsembleKalman2021a,hoelMultilevelEnsembleKalman2016a}; our description generalizes it to the framework in \cref{sec:frame:sl-mf}.

        Consider $\omega^{\ell,j}\in\Omega$ for $0\le\ell\le L$, $1\le j\le J_\ell$. The multilevel ensemble~evolves~as
        \begin{equation} \label{eq:frame:def:ml-evol}
            \begin{aligned}
                    u_{n+1}^{\ell,{\mathrm F},j} &= \Psi_n^{\mathcal G_\ell}(u_n^{\ell,{\mathrm F},j}, \widehat\Theta^{\mathrm{ML}}(\bm u_n^{\mathrm{ML}}), \xi_n^{\ell,j}), &\qquad 1 \le j \le J_\ell, &\quad 0 \le \ell \le L,\\
                    u_{n+1}^{\ell,{\mathrm C},j} &= \Psi_n^{\mathcal G_{\ell-1}}(u_n^{\ell,{\mathrm C},j}, \widehat\Theta^{\mathrm{ML}}(\bm u_n^{\mathrm{ML}}), \xi_n^{\ell,j}), &\qquad 1 \le j \le J_\ell, &\quad 1 \le \ell \le L,
            \end{aligned}
        \end{equation}
        where $\xi_n^{\ell,j}\coloneqq\xi_n(\omega^{\ell,j})$ and, analogously to \cref{eq:intro:mlmc:mlmc}, the multilevel sample statistic
        \begin{equation} \label{eq:frame:def:thetaml}
            \widehat\Theta^{\mathrm{ML}}(\bm u_n^{\mathrm{ML}}) \coloneqq \widehat\Theta^{\mathcal G_0}(\bm u_n^{0,{\mathrm F}}) + \sum\nolimits_{\ell=1}^L\Bigl(\widehat\Theta^{\mathcal G_\ell}(\bm u_n^{\ell,{\mathrm F}}) - \widehat\Theta^{\mathcal G_{\ell-1}}(\bm u_n^{\ell,{\mathrm C}})\Bigr)
        \end{equation}
        estimates $\Theta^{\mathcal G_L}$. Similarly to before, if $g$ is stochastic, $g(u_n^{\ell, \{F,C\},j})$ should be interpreted as $g(u_n^{\ell, \{F,C\},j}, \omega^{\ell,j})$ in the computation of $\Psi^g$ and $\widehat \Theta^g$. The fine-coarse particle pairs are correlated by setting $u_0^{\ell,{\mathrm F},j}=u_0^{\ell,{\mathrm C},j}$ and using the shared $\omega^{\ell,j}$. The QoI $\bar\theta^\dagger_N$ is estimated by a multilevel estimator ${\hat \theta^{\dagger,{\mathrm{ML}}}_N \coloneq \widehat \Theta^{\dagger,{\mathrm{ML}}}(\bm u_N^{\mathrm{ML}})}$ that is analogous to \cref{eq:frame:def:thetaml}. Note that the multilevel estimator (\ref{eq:frame:def:thetaml}) may not preserve properties of $\Theta$ such as definiteness. Dynamics can be adapted to deal with this complication; see \cref{rem:apdx-proofs:ipm:cov}.

\section{Theoretical properties and convergence} \label{sec:conv}
    \Cref{sec:conv:ass} formulates assumptions on the ingredients of the single- and multilevel framework outlined in \cref{sec:frame}. Under these assumptions, \cref{sec:conv:rates} gives convergence rates to the mean-field model for both simulation algorithms.

    \subsection{Assumptions} \label{sec:conv:ass}
        We formulate assumptions on the following ingredients of the framework: (i) the approximations $\mathcal G_\ell$ to the exact forward model $\mathcal G$, (ii) the functions $\Psi_n^g$ defined in \cref{sec:frame}, and (iii) the parameter $\Theta^g$ and its estimator $\widehat\Theta^g$. These assumptions are \emph{local}, and hence contain \emph{locality conditions} such as $\normZZZZ{u_1-u_0}_r\le d$.
        \begin{assumption} \label{ass:frame:def:G}
            There exist constants $\beta$ and $\gamma$ such that, for any $p\ge2$ and $u_0\in L^{\ge2}(\Omega, \mathbb R^{d_u})$, there exist constants $d, c_{g,\{1,2,3, 4\}}>0$ and $r\ge2$ such that for any $u_{\{1,2\}}\in L^{\ge2}(\Omega, \mathbb R^{d_u})$ with $\normZZZZ{u_{\{1,2\}}-u_0}_r\le d$, the following hold for $\ell\ge0$.
            \begin{enumerate}[label={(\roman*{})}]
                \item The models $\mathcal G_\ell$ satisfy a Lipschitz bound: ${\normZZZZ{\mathcal G_\ell(u_1) - \mathcal G_\ell(u_2)}_p \le c_{g,1}\normZZZZ{u_1-u_2}_r}$.
                \item All $\mathcal G_\ell$ are bounded: $\normZZZZ{\mathcal G_\ell(u_1)}_p \le c_{g,2}$.
                \item The rate of approximation to $\mathcal G$ is described by $\beta$: ${\normZZZZ{\mathcal G_\ell(u_1) - \mathcal G(u_1)}_p \le c_{g,3} 2^{-\beta \ell/2}}$.
                \item The rate at which $\mathcal G_\ell$ increases in cost is described by $\gamma$: $\mathrm{Cost}(\mathcal G_\ell) \le c_{g,4}2^{\gamma \ell}$.
            \end{enumerate}
        \end{assumption}
        
        \begin{assumption} \label{ass:frame:def:psi}
            For any $u_0\in L^{\ge2}(\Omega, \mathbb R^{d_u})$, $\theta_0\in L^{\ge2}(\Omega, \mathbb R^{d_\theta})$, and $p\ge2$, and with $\xi\sim\mathcal N(0, I)$, there exist constants $d,c_{\psi}>0$ and $r\ge2$ such that, for any $u_{\{1,2\}}\in L^{\ge2}(\Omega, \mathbb R^{d_u})$,  $\theta_1\in\mathbb{R}^{d_\theta}$, and $\theta_2\in L^{\ge2}(\Omega, \mathbb R^{d_\theta})$, then if $\normZZZZ{u_{\{1,2\}}-u_0}_r\le d$ and $\normZZZZ{\theta_{\{1,2\}}-\theta_0}_r\le d$, all functions $\Psi_n^g$ satisfy a local Lipschitz bound:
            \begin{equation*}
            \begin{aligned}
                &\normZZZZ{\Psi_n^{g_1}(u_1, \theta_1, \xi) - \Psi_n^{g_2}(u_2, \theta_2, \xi)}_p\\ 
                &\qquad \le c_{\psi}(\normZZZZ{u_1-u_2}_r + \normZZZZ{\theta_1-\theta_2}_r + \normZZZZ{g_1(u_1)-g_2(u_2)}_r).
            \end{aligned}
            \end{equation*}
        \end{assumption}
        
        \begin{assumption} \label{ass:frame:def:theta}
            For any $u_0\in L^{\ge2}(\Omega, \mathbb R^{d_u})$ and $p\ge2$, there exist constants $r\ge2$ and $d,c_{\theta,\{1,2,3,4\}}>0$ such that, for any $u_{\{1,2\}}\in L^{\ge2}(\Omega, \mathbb R^{d_u})$ with $\normZZZZ{u_{\{1,2\}}-u_0}_r\le d$, the following hold (with $\bm u_{\{1,2\}}$ an ensemble of $J$ particles distributed as $u_{\{1,2\}}$).
            \begin{subequations}
            \begin{enumerate}[label={(\roman*{})}]
                \item The statistic $\widehat\Theta^g$ satisfies a local Lipschitz bound:
                \begin{equation*}
                    \normZZZZ{\widehat\Theta^{g_1}(\bm u_1) - \widehat\Theta^{g_2}(\bm u_2)}_p \le c_{\theta,1}(\normZZZZ{u_1-u_2}_r + \normZZZZ{g_1(u_1)-g_2(u_2)}_r).
                \end{equation*}
                \item With $\bm u_i$ i.i.d., a difference in $\Theta^g$ is estimated by a difference in $\widehat\Theta^g$ with error
                \begin{align*}
                    &\normZZZZ{(\widehat\Theta^{g_1}(\bm u_1) - \widehat\Theta^{g_2}(\bm u_2)) - (\Theta^{g_1}[u_1] - \Theta^{g_2}[u_2])}_p\\&\qquad \le c_{\theta,2}J^{-1/2}(\normZZZZ{u_1-u_2}_r + \normZZZZ{g_1(u_1)-g_2(u_2)}_r).
                \end{align*}
                \item With $\bm u_i$ i.i.d., $\Theta^g$ is estimated by $\widehat\Theta^g$ with error ${\normZZZZ{\widehat\Theta^g(\bm u_1) - \Theta^g[u_1]}_p \le c_{\theta,3}J^{-1/2}}$.
                \item The parameter and statistic are bounded: $\absZZZZ{\Theta^g[u_1]} \le c_{\theta,4}$ and $\normZZZZ{\widehat\Theta^g(\bm u_1)}_p \le c_{\theta,4}$.
            \end{enumerate}
            \end{subequations}
            The properties formulated here for $\Theta^g$ and $\widehat\Theta^g$ must be satisfied by $\Theta^\dagger$ and $\widehat\Theta^\dagger$ as well.
        \end{assumption}
        \begin{remark} \label{rem:frame:def:theta}
            By setting $J=1$ in \cref{ass:frame:def:theta}(i--ii), another property emerges:
            \begin{equation}
                \absZZZZ{\Theta^{g_1}[u_1] - \Theta^{g_2}[u_2]} \le (c_{\theta,1}+c_{\theta,2})(\normZZZZ{u_1-u_2}_r + \normZZZZ{g_1(u_1)-g_2(u_2)}_r).
            \end{equation}
        \end{remark}
        \Cref{ass:frame:def:G} pertains to $\mathcal G$ and hence must be checked on a case-by-case basis. \Cref{ass:frame:def:psi,ass:frame:def:theta} are discussed for the dynamics used in this paper in \cref{sec:apdx-proofs}.

    \subsection{Convergence rates} \label{sec:conv:rates}
        Our main theorem bounds the asymptotic cost-to-error relation of the multilevel algorithm from \cref{sec:frame:ml}, when the number of levels and of particles on each level are chosen in a specified way. It generalizes \cite[Theorem 3.2]{hoelMultilevelEnsembleKalman2016a} from the EnKF case to our framework. We then give a slower single-level convergence rate for comparison.
        \begin{theorem} \label{thm:frame:ml}
            Let $\epsilon>0$. If \cref{ass:frame:def:G,ass:frame:def:psi,ass:frame:def:theta} are satisfied and
            \begin{equation} \label{eq:thm:frame:ml:choices}
                L = \floor{2\log_2(\epsilon^{-1})/\beta} \qquad \text{and} \qquad J_\ell \eqsim 2^{-\frac{\beta+2\gamma}3\ell}\left\{\rule{0cm}{0.7cm}\right.\begin{array}{ll}
                    2^{\beta L} & \text{if $\beta > \gamma$},\\
                    L^2 2^{\beta L} & \text{if $\beta = \gamma$},\\
                    2^{\frac{\beta+2\gamma}3L} & \text{if $\beta < \gamma$},
                \end{array}
            \end{equation}
            then for every $p\ge2$, there exists an $\epsilon_0>0$ such that
            \begin{equation} \label{eq:thm:frame:ml:res}
                \normZZZZ{\hat \theta^{\dagger,{\mathrm{ML}}}_N - \bar \theta^\dagger_N}_p \lesssim \epsilon\log_2(\epsilon^{-1})^N \qquad \text{when} \qquad \epsilon\le\epsilon_0
            \end{equation}
            with the multilevel simulation algorithm from \cref{sec:frame:ml}, for a cost
            \begin{equation} \label{eq:thm:frame:ml:cost}
                \mathrm{Cost} \eqsim \left\{\rule{0cm}{0.7cm}\right.\begin{array}{ll}
                    \epsilon^{-2} & \text{if $\beta > \gamma$},\\
                    \epsilon^{-(2+\delta)} & \text{if $\beta = \gamma$},\\
                    \epsilon^{-2\gamma/\beta} & \text{if $\beta < \gamma$},
                \end{array} \qquad \text{for any $\delta > 0$.}
            \end{equation}
        \end{theorem}
        \begin{proof}
            The proof is given in \cref{sec:proof-ml}.
        \end{proof}
        \begin{theorem} \label{lmm:frame:sl}
            Let $\epsilon > 0$. If \cref{ass:frame:def:G,ass:frame:def:psi,ass:frame:def:theta} are satisfied and
            \begin{equation} \label{eq:lmm:frame:sl:choices}
                L = \floor{2\log_2(\epsilon^{-1})/\beta} \qquad \text{and} \qquad J \eqsim \epsilon^{-2},
            \end{equation}
            then for every $p\ge2$, there exists an $\epsilon_0>0$ such that
            \begin{equation} \label{eq:lmm:frame:sl:res}
                \normZZZZ{\hat \theta_N^{\dagger,L} - \bar \theta^\dagger_N}_p \lesssim \epsilon \qquad \text{when} \qquad \epsilon\le\epsilon_0
            \end{equation}
            with the single-level simulation algorithm from \cref{sec:frame:sl-mf}, for a cost
            \begin{equation}
                \mathrm{Cost} \eqsim \epsilon^{-(2 + 2\gamma/\beta)}.
            \end{equation}
        \end{theorem}
        \begin{proof}
            The proof is given in \cref{sec:proof-sl}.
        \end{proof}
        \begin{remark}[On the extra factor in \cref{eq:thm:frame:ml:res}] \label{rem:conv:ml:factor}
            The factor $\log_2(\epsilon^{-1})^N$ in \cref{eq:thm:frame:ml:res} also appears in the bounds of \cite{hoelMultilevelEnsembleKalman2016a} and its follow-up work \cite{chernovMultilevelEnsembleKalman2021a}. Like us, they note that this factor does not manifest in numerical tests. This is important for the feasibility of the method: while the asymptotic effect of the factor is limited since $\log_2(\epsilon^{-1})^N\epsilon\lesssim\epsilon^{1-\delta}$ for all $\delta>0$, it would introduce an enormous constant when $N$ is moderate or large.
        \end{remark}

\section{Proof of \texorpdfstring{\cref*{thm:frame:ml}}{the multilevel convergence rate}} \label{sec:proof-ml}
    To prove \cref{thm:frame:ml}, we will make use of a number of auxiliary particles
    \begin{equation}
        \bar u_{n+1}^\ell(\omega) = \Psi_n^{\mathcal G_\ell(\cdot, \omega)}(\bar u_n^\ell(\omega), \Theta^{\mathcal G}[\bar u_n], \xi_n(\omega)).
    \end{equation}
    Note that this is not a McKean--Vlasov-type equation: the evolution depends on the law of $\bar u_n$, not of $\bar u_n^\ell$ itself. We define $\bar{\bm u}_n^{\mathrm{ML}}\coloneqq(\bar{\bm u}_n^{0,{\mathrm F}}, (\bar{\bm u}_n^{1,{\mathrm F}},\bar{\bm u}_n^{1,{\mathrm C}}), \ldots, (\bar{\bm u}_n^{L,{\mathrm F}},\bar{\bm u}_n^{L,{\mathrm C}}))$, where $\bar u_n^{\ell,{\mathrm F},j}\coloneqq \bar u_n^\ell(\omega^{\ell,j})$ and $\bar u_n^{\ell,{\mathrm C},j}\coloneqq\bar u_n^{\ell-1}(\omega^{\ell,j})$. These auxiliary particles will serve as a bridge between the mean-field and finite-ensemble particles: they use the interaction terms of the former, but the forward model of the latter.

    \Cref{ass:frame:def:G,ass:frame:def:psi,ass:frame:def:theta} contain locality conditions. For each induction step in the proofs below, one must ensure a sufficiently small $\epsilon$ (and, in some proofs, a sufficiently large $\ell$), such that the locality conditions with $u_0=\bar u_n$ and $\theta_0 = \Theta^{\mathcal G}[\bar u_n]$ will still hold \emph{after} having taken this step (allowing us to continue the induction). The fact that this is possible follows from the inequalities in the proofs. We call $\epsilon_0>0$ (and $\ell_0\ge0$) the smallest (and largest) of these values. This allows us to use \cref{ass:frame:def:G,ass:frame:def:psi,ass:frame:def:theta} whenever $\epsilon \le \epsilon_0$ (and $\ell \ge \ell_0$). In addition, assumptions \ref{ass:frame:def:G}(ii) and \ref{ass:frame:def:theta}(iv) will ensure that $u_1$, $u_2$, and $\theta_2$ in \cref{ass:frame:def:psi} will always be in $L^{\ge2}(\Omega, \mathbb R^{d_u})$ or $L^{\ge2}(\Omega, \mathbb R^{d_\theta})$, as required.

    For notational convenience, we allow ourselves to write $\Theta^{\mathcal{G}_{-1}} \coloneq 0$ and $\widehat \Theta^{\mathcal{G}_{-1}} \coloneq 0$.

    \begin{lemma} \label{lmm:apdx-proof-sl:bars}
        For all $n\ge 0$, $p \ge 2$, and $\ell\ge\ell_0$, it holds that $\normZZZZ{\bar u_n - \bar u_n^\ell}_p \lesssim 2^{-\beta\ell/2}$.
    \end{lemma}
    \begin{proof}
        For $n=0$, the statement definitely holds, as $\bar u_0 = \bar u_0^\ell$. We proceed by induction: if $\normZZZZ{\bar u_n - \bar u_n^\ell}_p \lesssim 2^{-\beta\ell/2}$ for all $p\ge 2$, then by applying assumptions \ref{ass:frame:def:psi} and \ref{ass:frame:def:G}(i, iii),
        \begin{equation*}
        \begin{aligned}
            \normZZZZ{\bar u_{n+1} - \bar u_{n+1}^\ell}_p &= \normZZZZ{\Psi_n^{\mathcal G}(\bar u_n, \Theta^{\mathcal G}[\bar u_n], \xi_n) - \Psi_n^{\mathcal G_\ell}(\bar u_n^\ell, \Theta^{\mathcal G}[\bar u_n], \xi_n)}_p\\
            &\leq c_{\psi}(\normZZZZ{\bar u_n - \bar u_n^\ell}_r + \normZZZZ{\mathcal G(\bar u_n) - \mathcal G_\ell(\bar u_n^\ell)}_r)\\
            &\leq c_{\psi}\bigl((1+c_{g,1})\normZZZZ{\bar u_n - \bar u_n^\ell}_{r'} + c_{g,3}2^{-\beta \ell/2}\bigr) \lesssim 2^{-\beta\ell/2}
        \end{aligned}
        \end{equation*}
        for all $p\ge2$. The last inequality holds due to the induction hypothesis.
    \end{proof}

    \begin{corollary} \label{cor:apdx-proof-sl:bars}
        From \cref{lmm:apdx-proof-sl:bars} and the triangle inequality follows $\normZZZZ{\bar u_n^{\ell+1} - \bar u_n^\ell}_p \lesssim 2^{-\beta \ell/2}$.
    \end{corollary}

    \begin{lemma} \label{lmm:apdx-proof-ml:bridge-theta}
        For all $n\ge0$ and $p\ge2$, when $\epsilon\le\epsilon_0$ it holds that
        \begin{equation}
            \normZZZZ{\widehat\Theta^{\mathrm{ML}}(\bm{\bar u}_n^{\mathrm{ML}}) - \Theta^{\mathcal G_L}[\bar u_n^L]}_p \lesssim \epsilon.
        \end{equation}
    \end{lemma}
    \begin{proof}
        We use the definition (\ref{eq:frame:def:thetaml}) of $\widehat\Theta^{\mathrm{ML}}$ to decompose $\normZZZZ{\widehat\Theta^{\mathrm{ML}}(\bm{\bar u}_n^{\mathrm{ML}}) - \Theta^{\mathcal G_L}[\bar u_n^L]}_p$, write $\Theta^{\mathcal G_L}[\bar u_n^L]$ as a telescoping sum, and then use the triangle inequality to get
        \begin{align*}
            &\normZZZZ[\Big]{\widehat\Theta^{\mathrm{ML}}(\bm{\bar u}_n^{\mathrm{ML}}) - \Theta^{\mathcal G_L}[\bar u_n^L]}_p \\&\qquad\le \sum\nolimits_{\ell=0}^L\normZZZZ[\Big]{\left(\widehat\Theta^{\mathcal G_\ell}(\bm{\bar u}_n^{\ell,{\mathrm F}}) - \widehat\Theta^{\mathcal G_{\ell-1}}(\bm{\bar u}_n^{\ell,{\mathrm C}})\right) - \left(\Theta^{\mathcal G_\ell}[\bar u_n^{\ell,{\mathrm F}}] - \Theta^{\mathcal G_{\ell-1}}[\bar u_n^{\ell,{\mathrm C}}]\right)}_p.
        \end{align*}
        Then, \cref{ass:frame:def:theta}(ii) bounds each term with $\ell\ge\ell_0$ by $c_{\theta,2}J_\ell^{-1/2}(\normZZZZ{\bar u_n^{\ell,{\mathrm F}} - \bar u_n^{\ell,{\mathrm C}}}_r + \normZZZZ{\mathcal G_\ell(\bar u_n^{\ell,{\mathrm F}})-\mathcal G_{\ell-1}(\bar u_n^{\ell,{\mathrm C}})}_r)$, while \cref{ass:frame:def:theta}(iii) ensures that the others are at most $2c_{\theta,3}J_\ell^{-1/2}$. By \cref{ass:frame:def:G}(i) and \cref{cor:apdx-proof-sl:bars}, we conclude that
        \begin{equation*}
        \begin{aligned}
            &\normZZZZ{\widehat\Theta^{\mathrm{ML}}(\bm{\bar u}_n^{\mathrm{ML}}) - \Theta^{\mathcal G_L}[\bar u_n^L]}_p \lesssim \sum\nolimits_{\ell=0}^{\ell_0-1}J_\ell^{-1/2} + \sum\nolimits_{\ell=\ell_0}^LJ_\ell^{-1/2}2^{-\beta \ell/2} 
            \\ &\qquad\le \Biggl(\ell_02^{(\beta+2\gamma)\ell_0/6} + \sum\nolimits_{\ell=0}^L{2^{(\gamma-\beta)\ell/3}}\Biggr)\left\{\rule{0cm}{0.7cm}\right.\begin{array}{ll}
                2^{-\beta L/2} & \text{if $\beta > \gamma$}\\
                L^{-1}2^{-\beta L/2} & \text{if $\beta = \gamma$}\\
                2^{-(\beta+2\gamma)L/6} & \text{if $\beta < \gamma$}
            \end{array}
            \\&\qquad\lesssim\epsilon.
        \end{aligned}
        \end{equation*}
    \end{proof}

    \begin{lemma} \label{lmm:apdx-proof-sl:aux}
        For all $n\ge0$, when $\epsilon\le\epsilon_0$ it holds that $\absZZZZ{\Theta^{\mathcal G_L}[\bar u_n^L] - \Theta^{\mathcal G}[\bar u_n]} \lesssim \epsilon$.
    \end{lemma}
    \begin{proof}
        Application of \cref{rem:frame:def:theta} yields
        \begin{equation*}
        \begin{aligned}
            \absZZZZ{\Theta^{\mathcal G_L}[\bar u_n^L] - \Theta^{\mathcal G}[\bar u_n]} &\le (c_{\theta,1} + c_{\theta,2})(\normZZZZ{\bar u_n^L - \bar u_n}_r + \normZZZZ{\mathcal G_L(\bar u_n^L) - \mathcal G(\bar u_n)}_r)\\
            &\le (c_{\theta,1} + c_{\theta,2})\left((1+c_{g,1})\normZZZZ{\bar u_n^L - \bar u_n}_{r'} + c_{g,3}2^{-\beta L/2}\right) \lesssim \epsilon,
        \end{aligned}
        \end{equation*}
        where the inequalities holds due to \cref{ass:frame:def:G}(i, iii), \cref{lmm:apdx-proof-sl:bars}, and \cref{eq:lmm:frame:sl:choices}.
    \end{proof}

    \begin{lemma} \label{lmm:apdx-proof-ml:double}
        For all $n\ge 0$, $0\le\ell\le L$, $1\le j\le J$, and $p \ge 2$, when $\epsilon \le \epsilon_0$, it holds that
        \begin{subequations} \label{eq:lmm:apdx-proof-ml:double}
        \begin{align}
            \label{eq:lmm:apdx-proof-ml:double:u} \normZZZZ{u_n^{\ell,\{{\mathrm F},{\mathrm C}\},j} - \bar u_n^{\ell,\{{\mathrm F},{\mathrm C}\},j}}_p &\lesssim \epsilon\log_2(\epsilon^{-1})^{n-1},\\
            \label{eq:lmm:apdx-proof-ml:double:theta} \normZZZZ{\widehat\Theta^{\mathrm{ML}}(\bm u_n^{\mathrm{ML}}) - \Theta^{\mathcal G}[\bar u_n]}_p &\lesssim \epsilon\log_2(\epsilon^{-1})^n,\\
            \label{eq:lmm:apdx-proof-ml:double:theta-dagger} \normZZZZ{\widehat\Theta^{\dagger,{\mathrm{ML}}}(\bm u_n^{\mathrm{ML}}) - \Theta^\dagger[\bar u_n]}_p &\lesssim \epsilon\log_2(\epsilon^{-1})^n.
        \end{align}
        \end{subequations}
    \end{lemma}
    \begin{proof}
        We first show that \cref{eq:lmm:apdx-proof-ml:double:theta,eq:lmm:apdx-proof-ml:double:theta-dagger} follow from \cref{eq:lmm:apdx-proof-ml:double:u} for any $n\ge0$:
        \begin{equation*}
        \begin{aligned}
            &\normZZZZ{\widehat\Theta^{\mathrm{ML}}(\bm u_n^{\mathrm{ML}}) - \Theta^{\mathcal G}[\bar u_n]}_p \lesssim \normZZZZ{\widehat\Theta^{\mathrm{ML}}(\bm u_n^{\mathrm{ML}}) - \widehat\Theta^{\mathrm{ML}}(\bm{\bar u}_n^{\mathrm{ML}})}_p + \epsilon + \epsilon\\
            &\quad = \normZZZZ[\Big]{\sum\nolimits_{\ell=0}^L\left(\widehat\Theta^{\mathcal G_\ell}(\bm u_n^{\ell,{\mathrm F}}) - \widehat\Theta^{\mathcal G_{\ell-1}}(\bm u_n^{\ell,{\mathrm C}})\right) - \left(\widehat\Theta^{\mathcal G_\ell}(\bm{\bar u}_n^{\ell,{\mathrm F}}) - \widehat\Theta^{\mathcal G_{\ell-1}}(\bm{\bar u}_n^{\ell,{\mathrm C}})\right)}_p + \epsilon + \epsilon\\
            &\quad= \normZZZZ[\Big]{\sum\nolimits_{\ell=0}^L\left(\widehat\Theta^{\mathcal G_\ell}(\bm u_n^{\ell,{\mathrm F}}) - \widehat\Theta^{\mathcal G_\ell}(\bm{\bar u}_n^{\ell,{\mathrm F}})\right) - \left(\widehat\Theta^{\mathcal G_{\ell-1}}(\bm u_n^{\ell,{\mathrm C}}) - \widehat\Theta^{\mathcal G_{\ell-1}}(\bm{\bar u}_n^{\ell,{\mathrm C}})\right)}_p + \epsilon + \epsilon\\
            &\quad\lesssim \log_2(\epsilon^{-1})\max_{0\le\ell\le L}\bigl(\normZZZZ{\bar u_n^{\ell,{\mathrm F},j} - u_n^{\ell,{\mathrm F},j}}_r + \normZZZZ{\bar u_n^{\ell,{\mathrm C},j} - u_n^{\ell,{\mathrm C},j}}_r\bigr) + \epsilon + \epsilon \lesssim \epsilon\log_2(\epsilon^{-1})^n
        \end{aligned}
        \end{equation*}
        and similarly for $\normZZZZ{\widehat\Theta^{\dagger,{\mathrm{ML}}}(\bm u_n^{\mathrm{ML}}) - \Theta^\dagger[\bar u_n]}_p$. For the first inequality, we used the triangle inequality and \cref{lmm:apdx-proof-ml:bridge-theta,lmm:apdx-proof-sl:aux}. For the second inequality, we used \cref{eq:thm:frame:ml:choices,ass:frame:def:theta}(i). For the last inequality, the terms are bounded due to \cref{eq:lmm:apdx-proof-ml:double:u}.

        We now prove \cref{eq:lmm:apdx-proof-ml:double} by induction. At $n=0$, \cref{eq:lmm:apdx-proof-ml:double:u} -- and therefore \cref{eq:lmm:apdx-proof-ml:double:theta,eq:lmm:apdx-proof-ml:double:theta-dagger} -- are clearly true. If \cref{eq:lmm:apdx-proof-ml:double} is satisfied at time $n$ for all $p\ge2$, then
        \begin{align*}
            &\normZZZZ{u_{n+1}^{\ell,\{{\mathrm F},{\mathrm C}\},j} - \bar u_{n+1}^{\ell,\{{\mathrm F},{\mathrm C}\},j}}_p\\
            &\quad \le c_{\psi}(1+c_{g,1})\normZZZZ{u_n^{\ell,\{{\mathrm F},{\mathrm C}\},j} - \bar u_n^{\ell,\{{\mathrm F},{\mathrm C}\},j}}_r + c_{\psi}\normZZZZ{\widehat\Theta^{\mathrm{ML}}(\bm u_n^{\mathrm{ML}}) - \Theta^{\mathcal G}[\bar u_n]}_r \lesssim \epsilon\log_2(\epsilon^{-1})^n
        \end{align*}
        proves \cref{eq:lmm:apdx-proof-ml:double:u} at time $n+1$ for all $p\ge2$, which again implies \cref{eq:lmm:apdx-proof-ml:double:theta,eq:lmm:apdx-proof-ml:double:theta-dagger}. Here, we used assumptions \ref{ass:frame:def:psi} and \ref{ass:frame:def:G}(i) and the induction hypothesis. Note that \cref{eq:thm:frame:ml:res} corresponds to \cref{eq:lmm:apdx-proof-ml:double:theta-dagger} at $n=N$.
    \end{proof}

\section{Proof of \texorpdfstring{\cref*{lmm:frame:sl}}{the single-level convergence rate}} \label{sec:proof-sl}
    For the proof in this section, we define $\bar u_n^j\coloneqq\bar u_n(\omega^j)$. This is a mean-field particle correlated to single-level particle $u_n^{L,j}$.

    \begin{lemma} \label{lmm:apdx-proof-sl:double}
        For all $n\ge 0$, $1\le j\le J$, and $p \ge 2$, when $\epsilon \le \epsilon_0$, it holds that
        \begin{subequations} \label{eq:lmm:apdx-proof-sl:double}
        \begin{align}
            \label{eq:lmm:apdx-proof-sl:double:u} \normZZZZ{u_n^{L,j} - \bar u_n^j}_p &\lesssim \epsilon,\\
            \label{eq:lmm:apdx-proof-sl:double:theta} \normZZZZ{\widehat\Theta^{\mathcal G_L}(\bm u_n^L) - \Theta^{\mathcal G}[\bar u_n]}_p &\lesssim \epsilon,\\
            \label{eq:lmm:apdx-proof-sl:double:theta-dagger} \normZZZZ{\widehat\Theta^\dagger(\bm u_n^L) - \Theta^\dagger[\bar u_n]}_p &\lesssim \epsilon.
        \end{align}
        \end{subequations}
    \end{lemma}
    \begin{proof}
        We first show that \cref{eq:lmm:apdx-proof-sl:double:theta,eq:lmm:apdx-proof-sl:double:theta-dagger} follow from \cref{eq:lmm:apdx-proof-sl:double:u} for any $n\ge0$:
        \begin{equation*}
        \begin{aligned}
            \normZZZZ{\widehat\Theta^{\mathcal G_L}(\bm u_n^L) - \Theta^{\mathcal G}[\bar u_n]}_p &\le \normZZZZ{\widehat\Theta^{\mathcal G_L}(\bm u_n^L) - \widehat\Theta^{\mathcal G}(\bm{\bar u}_n)}_p + \normZZZZ{\widehat\Theta^{\mathcal G}(\bm{\bar u}_n) - \Theta^{\mathcal G}[\bar u_n]}_p\\
            &\le c_{\theta,1}(\normZZZZ{u_n^{L,j} - \bar u_n^j}_r + \normZZZZ{\mathcal G_L(u_n^{L,j}) - \mathcal G(\bar u_n^j)}_r) + c_{\theta,3}J^{-1/2}\\
            &\le c_{\theta,1}\bigl((1+c_{g,1})\normZZZZ{u_n^{L,j} - \bar u_n^j}_{r'} + c_{g,3}2^{-\beta L/2}\bigr) + c_{\theta,3}J^{-1/2}\lesssim \epsilon\\
        \end{aligned}
        \end{equation*}
        and similarly for $\normZZZZ{\widehat\Theta^\dagger(\bm u_n^L) - \Theta^\dagger[\bar u_n]}_p$. For the second inequality, we used \cref{ass:frame:def:theta}(i, iii). For the third, we used \cref{ass:frame:def:G}(i, iii). For the last inequality, the terms are bounded due to \cref{eq:lmm:frame:sl:choices,eq:lmm:apdx-proof-sl:double:u}.
        
        We now prove \cref{eq:lmm:apdx-proof-sl:double} by induction. At $n=0$, \cref{eq:lmm:apdx-proof-sl:double:u} -- and therefore \cref{eq:lmm:apdx-proof-sl:double:theta,eq:lmm:apdx-proof-sl:double:theta-dagger} -- clearly holds, as $u_0^{L,j} = \bar u_0^j$. If \cref{eq:lmm:apdx-proof-sl:double} is satisfied at time $n$ for all $p\ge2$, then
        \begin{align*}
            \allowdisplaybreaks
            &\normZZZZ{u_{n+1}^{L,j} - \bar u_{n+1}^j}_p = \normZZZZ{\Psi_n^{\mathcal G_L}(u_n^{L,j}, \widehat\Theta^{\mathcal G_L}(\bm u_n^L), \xi_n^j) - \Psi_n^{\mathcal G}(\bar u_n^j, \Theta^{\mathcal G}[\bar u_n], \xi_n^j)}_p\\
            &\qquad\leq c_{\psi}\bigl((1+c_{g,1})\normZZZZ{u_n^{L,j} - \bar u_n^j}_r + c_{g,3}2^{-\beta L/2}\bigr) + c_{\psi}\normZZZZ{\widehat\Theta^{\mathcal G_L}(\bm u_n^L) - \Theta^{\mathcal G}[\bar u_n]}_r \lesssim \epsilon
        \end{align*}
        proves \cref{eq:lmm:apdx-proof-sl:double:u} at time $n+1$ for all $p\ge2$, which again implies \cref{eq:lmm:apdx-proof-sl:double:theta,eq:lmm:apdx-proof-sl:double:theta-dagger}. In the first inequality, we used assumptions \ref{ass:frame:def:psi} and \ref{ass:frame:def:G}(i, iii). In the last, we used the induction hypothesis and \cref{eq:lmm:frame:sl:choices}. Note that \cref{eq:thm:frame:ml:res} corresponds to \cref{eq:lmm:apdx-proof-sl:double:theta-dagger} at $n=N$.
    \end{proof}

\section{Scaling experiments} \label{sec:scale}
    In this section, we set out to corroborate our single-level and multilevel asymptotic cost-to-error bounds with state estimation of an Ornstein--Uhlenbeck process in \cref{sec:scale:c1} and Bayesian inversion of Darcy flow in \cref{sec:scale:c2}. Finally, we discuss and interpret our results in \cref{sec:scale:discussion}. Our code is open-source and can be found in the repository
    \begin{quote}
        \texttt{\url{https://gitlab.kuleuven.be/numa/public/paper-code-mlek}}.
    \end{quote}

    \subsection{State estimation for an Ornstein--Uhlenbeck process} \label{sec:scale:c1}
        The ensemble Kalman filter is extensively studied in \cite{hoelMultilevelEnsembleKalman2016a}, whose MLEnKF algorithm matches our method when applied to the EnKF. Hence, we study the DEnKF instead and apply it to estimate the state in the same Ornstein--Uhlenbeck process
        \begin{equation} \label{eq:scale:c1:ou}
            \mathrm d u = -u\mathrm dt+\sigma \mathrm dW_t, \qquad u(0) = 1,
        \end{equation}
        with $\sigma=0.5$, as in \cite{hoelMultilevelEnsembleKalman2016a}, from noisy measurements $y_n$ at $t_n = n$.
        
        \begin{figure}[htbp]
            \centering
            \def\svgwidth{.7\columnwidth}
            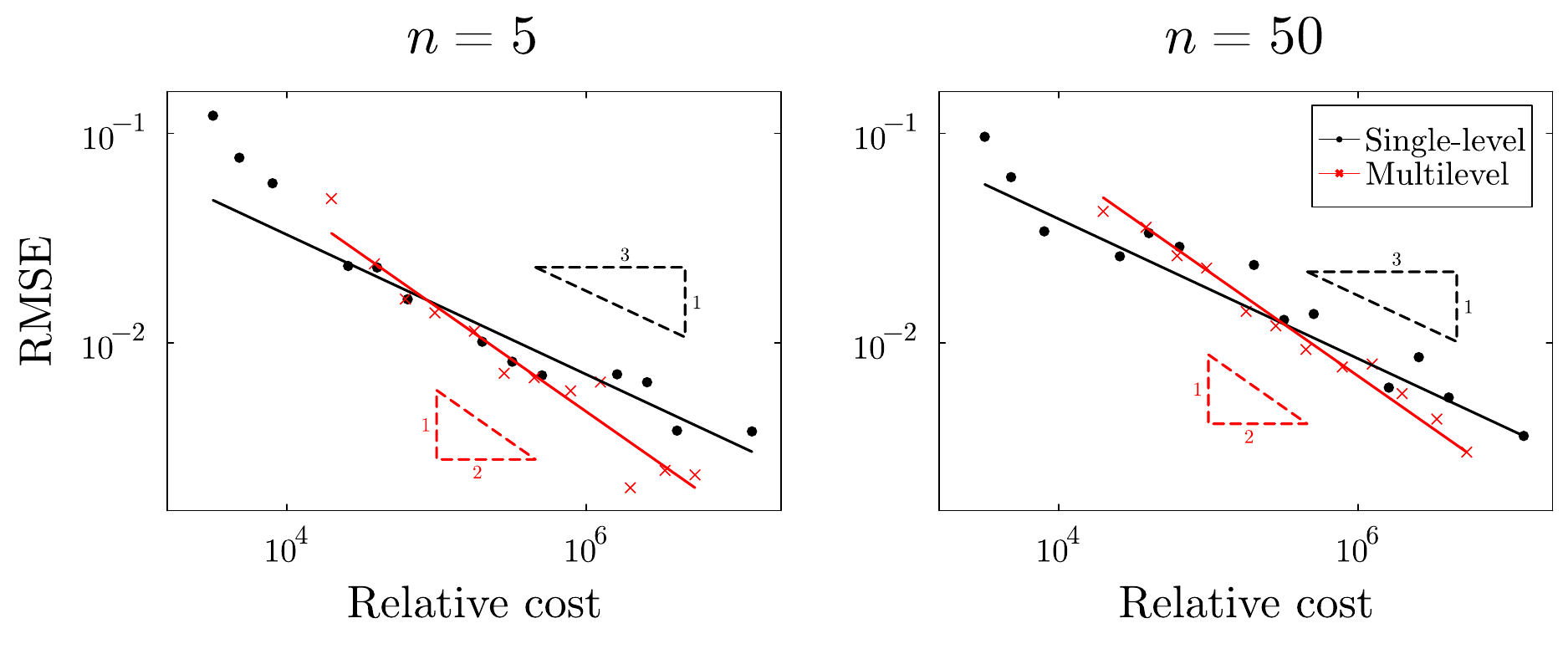
            \vspace{-.2cm}
            \caption{RMSE in function of computational cost for Ornstein--Uhlenbeck with DEnKF\vspace{-.6cm}}
            \label{fig:scale:c1:ou}
        \end{figure}
        
        Let us call $\mathcal G$ the operator that evolves \cref{eq:scale:c1:ou} exactly over a time interval of $\Delta t=1$. Then $u_{n+1} = \mathcal G(u_n)$ and $y_n = u_n + \eta_n$, with measurement noise $\eta_n\sim N(0,0.04I)$. For this example, we assume that evaluating $\mathcal G$ exactly is infeasible and introduce a hierarchy of Milstein discretizations of \cref{eq:scale:c1:ou} at resolutions $\Delta t_\ell = 2^{-\ell}$ as approximations $\{\mathcal G_\ell\}_\ell$. (In reality, $\mathcal G(u) = e^{-1}u + \xi$, with  $\xi\sim N(0, \frac{\sigma^2}{2}(1-e^{-2}))$ integrates \cref{eq:scale:c1:ou} exactly over $\Delta t=1$.)
        
        The cost to evaluate this approximation scales as $\mathrm{Cost}(\mathcal G_\ell) \eqsim 2^\ell$. Therefore, $\gamma=1$; we also have $\beta=2$. For the QoI of the filtering distribution, \pagebreak{}we choose the particle mean. Our gold standard is a highly accurate approximation of the expected value of the mean-field particle distribution obtained by averaging the means of $10$ large-scale ($J=10^4$) single-level simulations with the exact forward model. We consider the root-mean-square error (RMSE) between $10$ independent ensemble means and this gold standard. In \cref{fig:scale:c1:ou}, the RSME is plotted against the computational cost (the equivalent number of $\mathcal G_0$-evaluations) for a number of both single-level and multilevel experiments. The values of $L$ and $J_\ell$ are chosen according to \cref{eq:thm:frame:ml:choices}. From $\beta$ and $\gamma$, by \cref{thm:frame:ml}, we predict $\mathrm{RMSE}\lesssim (\mathrm{Cost}^\mathrm{SL})^{-1/3}$ for single-level and $\mathrm{RMSE}\lesssim(\mathrm{Cost}^\mathrm{ML})^{-1/2}$ for multilevel. The figure shows that both schemes are asymptotically close to these rates.

    \subsection{Bayesian inversion for Darcy flow} \label{sec:scale:c2}
        We now apply EKI and EKS to the inversion of a Darcy flow problem.

        \begin{remark}
            Some ensemble Kalman methods for Bayesian inversion, such as EKI and EKS, are based on the mean-field particle distribution for $N\rightarrow\infty$. Hence, we add a dependence of the number of time steps $N$ on the parameter $\epsilon$ (which is proportional to the error): $N = N(\epsilon)$. In this subsection, we will assume that $\normZZZZ{\hat\theta_N^{\dagger,\mathrm{ML}} - \bar \theta^\dagger_N}_p \lesssim \epsilon$ without the $\log_2(\epsilon^{-1})^N$ factor (see \cref{rem:conv:ml:factor}) and that the constant implicit in $\lesssim$ is $N$-independent. Our numerical investigations support this second assumption. We note that the theoretical bounds in multilevel particle filters \cite{jasraMultilevelParticleFilters2017} also allow time-dependent errors, while numerical tests show time-uniformity.
        \end{remark}
        
        Darcy flow is a classical test problem in Bayesian inversion (see, e.g., \cite{garbuno-inigoInteractingLangevinDiffusions2020a,huangIteratedKalmanMethodology2022b}). On the two-dimensional spatial domain $[0, 1]^2$, our forward model $\mathcal G(u)$ computes the~map~of the permeability field $a(x, u)$ of a porous medium to the pressure field $p(x)$ that satisfies
        \begin{equation} \label{eq:scale:c2:darcy}
            -\nabla\cdot(a(x, u)\nabla p(x)) = f(x) \qquad \text{and} \qquad p(x) = 0 \text{ if $x\in\partial[0, 1]^2$.}
        \end{equation}
        We set $f(x) = 1000\exp(x_1+x_2)$ and model $a(x, u)$ as a log-normal random field with covariance $(-\Delta+\tau^2)^{-d}$ with $d=2$ and $\tau=3$. This corresponds to a standard normal prior on the parameters $u_k$, the $d_u=16$ coefficients with largest eigenvalues in the Karhunen--Lo\`eve expansion
        \begin{equation}
            \log a(x, u) = \sum\nolimits_{k\in \mathbb N^2\backslash\{(0, 0)\}} u_k\sqrt{\lambda_k}\phi_k(x),
        \end{equation}
        with eigenpairs $\lambda_k = (\pi^2\|k\|_2^2 + \tau^2)^{-d}$ and $\phi_k(x) = c_k\cos(\pi k_1x_1)\cos(\pi k_2x_2)$, with $c_k=\sqrt2$ if $k_1k_2=0$ \pagebreak{}and $c_k=2$ otherwise (see \cite{huangIteratedKalmanMethodology2022b}). The output of the model consists of $p(x)$ evaluated on 49 equispaced points and the added noise has covariance $\Gamma=0.01I$. For EKS, we use a zero-centered Gaussian prior with covariance $\Gamma_0=I$.
        
        To stabilize and accelerate both the single- and the multilevel methods, we use the adaptive time steps introduced in \cite{kovachkiEnsembleKalmanInversion2019a} and also used in \cite{garbuno-inigoInteractingLangevinDiffusions2020a}. We compute them for each level independently, and propagate the whole ensemble with the minimal value. Although the adaptive versions of both methods converge exponentially \cite{garbuno-inigoInteractingLangevinDiffusions2020a}, we choose the conservative $N(\epsilon) \eqsim \epsilon^{-\delta}$ (with $\delta=0.1$).
        
        Solving \cref{eq:scale:c2:darcy} via central finite differences on a grid with step size $h_\ell \eqsim 2^{(13+\ell)/4}$ yields a hierarchy of forward models $\{\mathcal G_\ell\}_\ell$ with $\beta=1$ and $\gamma = 1/2$. \Cref{fig:scale:c2:darcy} shows good correspondence between experimental results and the expected rates.
        \begin{figure}[t]
            \centering
            \def\svgwidth{.7\columnwidth}
            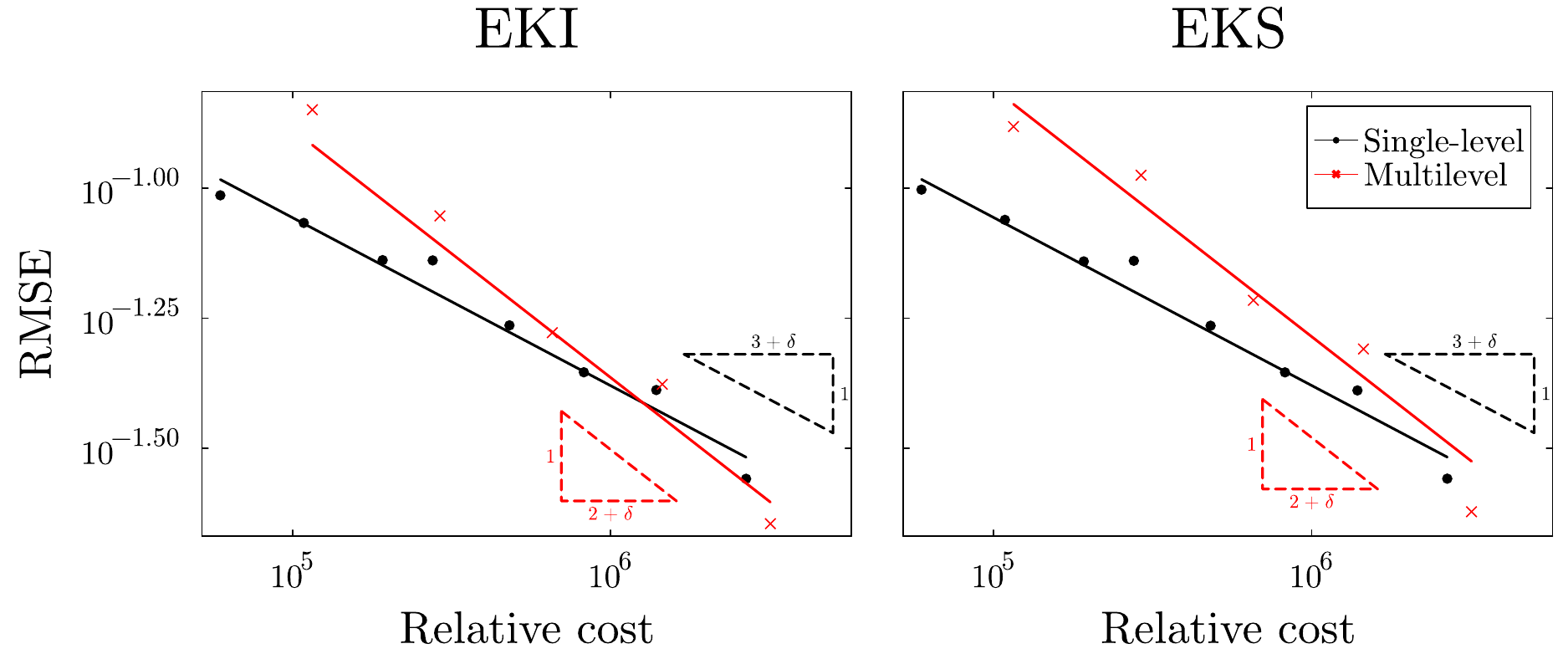
            \caption{RMSE in function of computational cost for Bayesian inversion of Darcy flow.\vspace{-.2cm}}
            \label{fig:scale:c2:darcy}
        \end{figure}

    \subsection{Discussion} \label{sec:scale:discussion}
        \Cref{lmm:frame:sl,thm:frame:ml} imply that our multilevel algorithm asymptotically enjoys faster convergence than the more straightforward single-level algorithm for a fixed number of time steps. The numerical scaling tests conducted in \cref{sec:scale:c1,sec:scale:c2} demonstrate that these rates are accurate. We stress that our theory only provides asymptotic rates and does not result in guidelines for selecting the constants in \cref{eq:lmm:frame:sl:choices,eq:thm:frame:ml:choices}.

        Changing these constants should, in general, shift the convergence graphs without changing their slopes. We cannot currently compare the non-asymptotic performance of single- and multilevel simulation in a meaningful way. However, the convergence rates can be compared, and those are the focus of this article. These limitations are also present in the experiments that are performed in \cite{chernovMultilevelEnsembleKalman2021a,hoelMultilevelEnsembleKalman2016a}. This important aspect of multilevel ensemble Kalman methods would be an interesting avenue for future research.

\section{Conclusions} \label{sec:concl}
    We have proposed a framework for simulating ensemble Kalman methods, including EnKF, DEnKF, EKI, and EKS. We consider the setting where the forward model $\mathcal G$ is intractable and is replaced by an approximation hierarchy $\{\mathcal G_\ell\}_\ell$.

    For methods in the framework, we have described and analyzed two algorithms that allow a numerical approximation of the particle density: a standard Monte Carlo simulation and a generalization of the multilevel Monte Carlo ensemble Kalman filter in \cite{chernovMultilevelEnsembleKalman2021a,hoelMultilevelEnsembleKalman2016a}, which we\pagebreak{} call the \emph{single-ensemble} multilevel algorithm. It uses a single, globally coupled ensemble whose particles use different forward models $\mathcal G_\ell$. The MLMC methodology is then used to estimate the interaction term at each time step.

    We formulated \cref{ass:frame:def:G,ass:frame:def:psi,ass:frame:def:theta}, under which convergence results are shown in \cref{thm:frame:ml,lmm:frame:sl}. These bounds suggest that single-ensemble MLMC asymptotically outperforms standard Monte Carlo when simulating until a fixed time step $N$. Experiments suggest that the convergence rates implied by the bounds are accurate.

    Open questions remain. Convergence bounds with explicit dependence on $N$ and potentially without the factor $\log_2(\epsilon^{-1})^N$ are crucial to better understand non-asymptotic properties of single- and multilevel ensemble Kalman methods. This is especially important for EKI and EKS, where the number of time steps depends on the desired accuracy. It would also be instructive to develop multiple-ensemble MLMC strategies for ensemble Kalman methods, such as the one proposed by \cite{hoelMultilevelEnsembleKalman2020b} for the EnKF, and compare to our method. In addition, this paper is a first step towards multilevel simulation methods for other particle-based methods for Bayesian inversion, such as consensus-based methods \cite{carrilloConsensusBasedSampling2022a,pinnauConsensusbasedModelGlobal2017}, whose interaction terms are more complex but still have cost $\mathcal O(J)$.

\appendix
\section{Verification of the assumptions for specific methods} \label{sec:apdx-proofs}
    This appendix verifies \cref{ass:frame:def:psi,ass:frame:def:theta} for the dynamics of the (deterministic) ensemble Kalman filter, ensemble Kalman inversion, and ensemble Kalman sampling.

    \begin{lemma} \label{lmm:apdx-proofs:theta:E}
        When the parameter $\Theta^g[\cdot]$ is the expectation $\mathbb E[\cdot]$ and the sample statistic $\widehat\Theta^g(\cdot)$ is the sample mean $E(\cdot)$, \cref{ass:frame:def:theta} is satisfied.
    \end{lemma}
    \begin{proof}
        We have $\normZZZZ{\widehat\Theta^g(\bm u_1) - \widehat\Theta^g(\bm u_2)}_p = \normZZZZ{\frac1J\sum\nolimits_{j=1}^J(u_1^j-u_2^j)}_p \le \normZZZZ{u_1-u_2}_p$ for any $p\ge2$, for \cref{ass:frame:def:theta}(i). For (ii), we have that $\normZZZZ{E(\bm u_1) - E(\bm u_2) - (\mathbb E[u_1] - \mathbb E[u_2])}_p$ equals $\normZZZZ{E((\bm u_1 - \bm u_2) - \mathbb E[u_1 - u_2])}_p$. Since a sample mean is an unbiased estimator, we obtain $\normZZZZ{E((\bm u_1 - \bm u_2) - \mathbb E[u_1 - u_2])}_p \le 2c_{\kern-1ptp}J^{-1/2}\normZZZZ{u_1 - u_2}_p$ with \cref{pr:intro:notation:mz,pr:intro:notation:ord}. Then, (iii) follows from (ii) by letting $u_2 \sim \delta_0$, then $\mathbb E[u_2] = 0$ and $c_{\theta,2}$ is $d$-independent. Finally, (iv) can be checked using \cref{pr:intro:notation:ord,pr:intro:notation:mz}.
    \end{proof}
    \begin{lemma}
        When the parameter $\Theta^g[\cdot]$ is the (cross-)covariance $\mathbb C[\cdot]$, $\mathbb C[g(\cdot)]$, or $\mathbb C[\cdot, g(\cdot)]$, and the sample statistic $\widehat\Theta^g(\cdot)$ is the corresponding sample (cross-)covariance $C(\cdot)$, $C(g(\cdot))$, or $C(\cdot, g(\cdot))$, \cref{ass:frame:def:theta} is satisfied.
    \end{lemma}
    \begin{proof}
        First consider the covariance $\mathbb C[(\cdot, g(\cdot))]$ and its corresponding sample covariance. For that estimator, the proof of (i) is contained in that of \cite[Lemma 3.9]{hoelMultilevelEnsembleKalman2016a}; the proof of (ii) in that of \cite[Lemma 3.8]{hoelMultilevelEnsembleKalman2016a}. Then, (iii) and (iv) can be checked as for \cref{lmm:apdx-proofs:theta:E}. The covariances in this lemma are submatrices of $\mathbb C[(\cdot, g(\cdot))]$; hence,~the~lemma~follows.
    \end{proof}

    \begin{remark}[Covariance matrices] \label{rem:apdx-proofs:ipm:cov}
        Dynamics such as the EnKF and EKS use covariance matrices for interaction. A mean-field or single-level sample covariance is always positive semi-definite, ensuring that operations such as matrix square roots are well-defined. A \emph{multilevel} sample covariance matrix (\ref{eq:frame:def:thetaml}), on the other hand, might have negative eigenvalues. In \cite{chernovMultilevelEnsembleKalman2021a,hoelMultilevelEnsembleKalman2016a}, this is avoided by setting all negative eigenvalues to zero in a preprocessing step. We will follow this approach and, to this end, define the operator
        \begin{equation} \label{eq:apdx-proofs:I+}
            I^+(M) = \sum\nolimits_{\lambda_k\ge0}\lambda_kq_kq_k^{\mathstrut\scriptstyle{\top}} \qquad \text{with $\{(\lambda_k, q_k)\}_k$ the eigenpairs of $M$},
        \end{equation}
        which can be freely incorporated into existing dynamics where needed. Indeed, when applied to sample covariance matrices in the single-level algorithm, it is the identity operator. In the multilevel context, in ensures positive semi-definiteness.
    \end{remark}

    \begin{lemma} \label{lmm:apdx-proofs:enkf}
        The ensemble Kalman filter and non-adaptive ensemble Kalman inversion fit into the framework of \cref{sec:frame:sl-mf} and satisfy \cref{ass:frame:def:psi}.
    \end{lemma}
    \begin{proof}
        We tackle the EnKF first. In the notation of our framework, we have $\widehat\Theta^{g}(\cdot)=C(g(\cdot))$ and the EnKF uses the functions $\Psi_n^g(u, \theta, \xi) = g(u) + K(y_{n+1}-Hg(u)+\sqrt\Gamma\xi)$, with $K = \theta H^{\mathstrut\scriptstyle{\top}}(HI^+(\theta) H^{\mathstrut\scriptstyle{\top}} + \Gamma)^{-1}$ and $I^+$ from \cref{rem:apdx-proofs:ipm:cov}. By \cref{pr:intro:notation:holder},
        \begin{equation} \label{eq:lmm:apdx-proofs:enkf:step1}
        \begin{aligned}
            \normZZZZ{\Psi_n^{g_1}(u_1, \theta_1, \xi) - \Psi_n^{g_2}(u_2, \theta_2, \xi)}_p &\le (1 + \absZZZZ{K_1}\absZZZZ{H})\normZZZZ{g_1(u_1)-g_2(u_2)}_p \\&\;\,+ \normZZZZ{K_1-K_2}_{2p}\normZZZZ{y_{n+1}-Hg_2(u_2)+\sqrt\Gamma\xi}_{2p}.
        \end{aligned}
        \end{equation}
        Now notice that $\absZZZZ{K_1}\le\absZZZZ{\theta_1}\absZZZZ{H}/\gamma_{\min}$, where $\gamma_{\min}>0$ is $\Gamma$'s smallest eigenvalue. In addition, \cite[Lemmas 3.3 and 3.4]{hoelMultilevelEnsembleKalman2016a} together prove, in our notation, the bound $\normZZZZ{K_1-K_2}_{2p}\le\absZZZZ{H}/\gamma_{\min}(1 + 2\absZZZZ{K_1}\absZZZZ{H})\normZZZZ{\theta_1-\theta_2}_{2p}$. Thus \cref{eq:lmm:apdx-proofs:enkf:step1} is bounded by
        \begin{equation*}
            \allowdisplaybreaks
            \Bigl(1+\frac{\absZZZZ H^2}{\gamma_{\min}}\absZZZZ{\theta_1}\Bigr)\normZZZZ{g_1(u_1)-g_2(u_2)}_p + \Bigl(\frac{\absZZZZ H}{\gamma_{\min}}+2\frac{\absZZZZ H^3}{\gamma^2_{\min}}\absZZZZ{\theta_1}\Bigr)\normZZZZ{y_{n+1}-Hg_2(u_2)+\sqrt\Gamma\xi}_{2p}\normZZZZ{\theta_1-\theta_2}_{2p};
        \end{equation*}
        bounding factors with the locality conditions (e.g., $\absZZZZ{\theta_1}\le\normZZZZ{\theta_0}+d$) and \cref{ass:frame:def:G}(ii) concludes the proof for EnKF. The EKI dynamics are an instance of the EnKF dynamics, with an enlarged state space \cite{iglesiasEnsembleKalmanMethods2013a}. The proof still applies.
    \end{proof}
    In the rest of this section, the use of \cref{pr:intro:notation:holder} and the final step of using the locality conditions and \cref{ass:frame:def:G}(ii) will be left implicit.

    \begin{lemma} \label{lmm:apdx-proofs:denkf}
        The deterministic ensemble Kalman filter fits into the framework of \cref{sec:frame:sl-mf} and satisfies \cref{ass:frame:def:psi}.
    \end{lemma}
    \begin{proof}
        With $\widehat\Theta^g(\cdot) = (E(g(\cdot)), C(g(\cdot)))$, it follows from \cref{eq:ex:intro:ipm:denkf} that $\Psi_n^g(u,\theta,\xi) = g(u) + K\bigl(y_{n+1} - H/2(g(u) + \theta^{(1)})\bigr)$, with $K\coloneqq\theta^{(2)}H^{\mathstrut\scriptstyle{\top}}(HI^+(\theta^{(2)})H^{\mathstrut\scriptstyle{\top}}+\Gamma)^{-1}$. Then
        \begin{align*}
            &\normZZZZ{\Psi_n^{g_1}(u_1, \theta_1, \xi) - \Psi_n^{g_2}(u_2, \theta_2, \xi)}_p\le(1+\absZZZZ{K_1}\absZZZZ{H}/2)\normZZZZ{g_1(u_1)-g_2(u_2)}_p\\&\qquad+ \normZZZZ{K_1-K_2}_{2p}\normZZZZ{(y_{n+1}-H/2(g_2(u_2)-\theta_2^{(1)}))}_{2p} + \absZZZZ{K_1}\absZZZZ{H}/2\normZZZZ{\theta_1^{(1)}-\theta_2^{(1)}}_p.
        \end{align*}
        The proof can then be finished similarly to that of \cref{lmm:apdx-proofs:enkf}.
    \end{proof}

    Dynamics such as EKS use the square root of the covariance matrix. To prove \cref{ass:frame:def:psi}, we will need the matrix square root to be Lipschitz continuous.
    \begin{lemma} \label{lmm:apdx-proofs:ipm:sqrtm}
        From \cite[Corollary 4.2]{vanhemmenInequalityTraceIdeals1980} and \cref{pr:intro:notation:mono} follows that, if $M_1\succeq\mu I$ with $\mu>0$ and $M_2\succeq0$, then $\normZZZZ{\sqrt{M_1} - \sqrt{M_2}}_p \le \sqrt{2/\mu}\normZZZZ{M_1-M_2}_p$ for all $p\ge2$.
    \end{lemma}
    \begin{lemma} \label{lmm:apdx-proofs:eks}
        Non-adaptive ensemble Kalman sampling fits into the framework of \cref{sec:frame:sl-mf} and, if there exists a $\mu>0$ such that $\mathbb C[\bar u_n]\succeq \mu I$ for all $n>0$, satisfies \cref{ass:frame:def:psi}.
    \end{lemma}
    \begin{proof}
        With $\widehat\Theta^g(\cdot) = (C(\cdot), C(\cdot, g(\cdot)))$, \cref{eq:ex:intro:ipm:eks} can be manipulated into
        \begin{equation}
            \allowdisplaybreaks
            \Psi_n^g(u, \theta, \xi) = (I + \tau_n\theta^{(1)}\Gamma_0^{-1})^{-1}\bigl[u + \tau_n\theta^{(2)}\Gamma^{-1}(y - g(u))\bigr] + \sqrt{\rule{0pt}{1.75ex}2\tau_n\smash[t]{I^+(\theta^{(1)})}}\kern2pt\xi.
        \end{equation}
        Then follows (as $\smash[t]{\normZZZZ{\theta_2^{(1)}-I^+(\theta_2^{(1)})}_p\le\normZZZZ{\theta_2^{(1)} - \theta_1^{(1)}}_p}$ similarly to \cite[Lemma 3.3]{hoelMultilevelEnsembleKalman2016a}):\pagebreak{}
        \begin{align*}
            \allowdisplaybreaks
            &\normZZZZ{\Psi_n^{g_1}(u_1, \theta_1, \xi) - \Psi_n^{g_2}(u_2, \theta_2, \xi)}_p \le \normZZZZ{\sqrt{2\tau_n}\kern2pt\xi}_{2p}\,\normZZZZ{\sqrt{\rule{0pt}{1.9ex}\smash[tb]{\theta_1^{(1)}}} - \sqrt{\rule{0pt}{1.9ex}\smash[tb]{I^+(\theta_2^{(1)})}}}_{2p}\\
            &\qquad+\normZZZZ{(I+\tau_n\theta_2^{(1)}\Gamma_0^{-1})^{-1}[u_1-u_2 + \tau_n\theta_1^{(2)}\Gamma^{-1}(y-g_1(u_2))-\tau_n\theta_2^{(2)}\Gamma^{-1}(y-g_2(u_2))]}_p\\
            &\qquad+\normZZZZ{((I+\tau_n\theta_1^{(1)}\Gamma_0^{-1})^{-1} - (I + \tau_n\theta_2^{(1)}\Gamma_0^{-1})^{-1})[u_1+\tau_n\theta_1^{(2)}\Gamma^{-1}(y-g_1(u_1))]}_p\\
            &\quad\le2\sqrt{2\tau_n/\mu}\normZZZZ{\xi}_{2p}\normZZZZ{\theta_1^{(1)}-\theta_2^{(1)}}_{2p}\\
            &\qquad + \normZZZZ{u_1-u_2}_p + \tau_n\absZZZZ{\theta_1^{(2)}\Gamma^{-1}}\,\normZZZZ{g_1(u_1)-g_2(u_2)}_p+\tau_n\absZZZZ{\Gamma^{-1}}\normZZZZ{(y-g_2(u_2))}_{2p}\normZZZZ{\theta_1^{(2)}-\theta_2^{(2)}}_{2p}\\
            &\qquad + \normZZZZ{u_1+\tau_n\theta_1^{(2)}\Gamma^{-1}(y-g_1(u_1))}_{2p}\tau_n\absZZZZ{\Gamma_0^{-1}}\normZZZZ{\theta_1^{(1)}-\theta_2^{(1)}}_{2p}.
        \end{align*}
        The last inequality bounded $\absZZZZ{(I+\tau_n\theta_2^1\Gamma_0^{-1})^{-1}}\le1$ and used the Lipschitz inequality that, for positive semi-definite matrices $A_i$, it holds that $\absZZZZ{(I+A_1)^{-1}-(I+A_2)^{-1}} = \absZZZZ{(I+A_1)^{-1}(A_2-A_1)(I+A_2)^{-1}} \le \absZZZZ{A_1-A_2}$. We were also able to use \cref{lmm:apdx-proofs:ipm:sqrtm} since, in \cref{sec:proof-sl,sec:proof-ml}, $\theta_1$ is always the mean-field parameter $\Theta^{\mathcal G}[\bar u_n]$.
    \end{proof}

\section*{Acknowledgments}
    We thank Ignace Bossuyt and Pieter Vanmechelen for their thorough reviews and helpful comments. Part of this work was financed by the Fonds Wetenschappelijk Onderzoek -- Vlaanderen (FWO) under grants 1169725N and 1SE1225N, and by the \mbox{European} High-Performance Computing Joint Undertaking (JU) under grant agreement No.\ 955701 (TIME-X). The JU receives support from the European Union's Horizon 2020 research and innovation programme and from Belgium, France, Germany, and Switzerland.

\bibliographystyle{abbrv}
\bibliography{references}

\end{document}